\documentclass[12pt,a4paper]{amsart}
\usepackage{amsmath,amssymb,amsfonts, float}
\usepackage[all]{xy}
\usepackage{caption}
\usepackage{enumerate}
\usepackage{mathpazo}
\usepackage{a4wide}

\usepackage{bm}
\usepackage{graphicx}
\usepackage[breaklinks=true]{hyperref}

\usepackage{xcolor}
\usepackage{multicol}
\usepackage{tabularx}

\usepackage{hyperref}
\usepackage[capitalize]{cleveref}

\usepackage[normalem]{ulem}

\newtheorem{thm}{Theorem}[section] 
\newtheorem*{thm*}{Theorem} 
\newtheorem{prop}[thm]{Proposition}
\newtheorem{lem}[thm]{Lemma}
\newtheorem{cor}[thm]{Corollary}

\theoremstyle{definition}
\newtheorem{definition}[thm]{Definition}
\newtheorem{expl}[thm]{Example}

\newtheorem{rem}[thm]{Remark}

\DeclareMathOperator{\C}{\mathbb{C}}
\DeclareMathOperator{\Z}{\mathbb{Z}}

\DeclareMathOperator{\F}{\mathbb{F}}

\DeclareMathOperator{\Hom}{{\rm Hom}}

\DeclareMathOperator{\OO}{\mathcal{O}}

    \DeclareFontFamily{U}{wncy}{}
    \DeclareFontShape{U}{wncy}{m}{n}{<->wncyr10}{}
    \DeclareSymbolFont{mcy}{U}{wncy}{m}{n}
    \DeclareMathSymbol{\Sha}{\mathord}{mcy}{"58}

\numberwithin{equation}{section}

\DeclareSymbolFont{bbold}{U}{bbold}{m}{n}
\DeclareSymbolFontAlphabet{\mathbbold}{bbold}

\usepackage{hyperref}

\newcommand{\Ann}{{\rm Ann}}

\newcommand{\I}{\mathcal{I}}

\title{On gcd-graphs over finite rings}
 \author{Tung T. Nguyen, Nguy$\tilde{\text{\^{e}}}$n Duy T\^{a}n }

 \address{Department of Mathematics and Computer Science, Lake Forest College, Lake Forest, Illinois, USA}
 \email{tnguyen@lakeforest.edu}
 
  \address{
Faculty Mathematics and Informatics, Hanoi University of Science and Technology, 1 Dai Co Viet Road, Hanoi, Vietnam } 
\email{tan.nguyenduy@hust.edu.vn}

\thanks{TTN is partially supported by an AMS-Simons Travel Grant.  NDT is partially supported by the Vietnam National
Foundation for Science and Technology Development (NAFOSTED) under grant number 101.04-2023.21}
\keywords{Gcd graphs, Finite rings, Ramanujan sums, Integral graphs.}
\subjclass[2020]{Primary 05C25, 11L05, 13A70, 13M05}
\begin{document}

\maketitle

\begin{abstract}
Gcd-graphs represent an interesting and historically important class of integral graphs. Since the pioneering work of Klotz and Sander, numerous incarnations of these graphs have been explored in the literature. In this article, we define and establish some foundational properties of gcd-graphs defined over a general finite commutative ring. In particular, we investigate the connectivity and diameter of these graphs. Additionally, when the ring is a finite symmetric $\Z/n$-algebra, we give an explicit description of their spectrum using the theory of Ramanujan sums that gives a unified treatment of various results in the literature. 

\end{abstract}
\maketitle
\section{Introduction and motivations}
Gcd-graphs are an interesting and historically important class of integral graphs; i.e., graphs whose eigenvalues are integers. These graphs are first introduced by Klotz and Sander in \cite{klotz2007some}. To set the context, let us briefly recall their definition. Let $n$ be a positive integer and $D$ a subset of proper divisors of $n$. The gcd-graph $G_n(D)$ is defined as follows: (1) The vertices of $G_n(D)$ are elements of the finite ring $\Z/n$ and (2) two vertices $a,b$ are adjacent if $\gcd(a-b, n) \in D.$ Using the theory of Ramanujan sums, one can describe the spectrum of the gcd-graph $G_n(D)$  explicitly (see \cite[Section 4]{klotz2007some}). More precisely, its eigenvalues are indexed by elements of $\Z/n$; namely $(\lambda_m)_{m \in \Z/n}$ where 
\begin{equation} \label{eq:ramanjuan_1}
\lambda_m = \sum_{d \in D} c(m,n/d),
\end{equation}
and 
\begin{equation} \label{eq:ramanujan_2}
c(m, n/d)=  \mu(t) \dfrac{\varphi(n/d)}{\varphi(t)}, \quad \text{where} \quad t = \dfrac{n/d}{\gcd(n/d,m)}.
\end{equation}
Here $\mu$ and $\varphi$ are respectively the M\"obius and Euler totient functions. A direct corollary of this explicit description is that gcd-graphs are integral. In \cite{so2006integral}, So goes one step further: he shows that gcd-graphs are the only integral circulant graphs. In \cite{minavc2024gcd}, inspired by the analogy between number fields and function fields, we study gcd-graphs over a polynomial ring with coefficients in a finite field. We show that in this case, there is a direct analog of Ramanujan sums that allows us to describe the spectrum of these gcd-graphs by an explicit formula similar to \cref{eq:ramanjuan_1}. Furthermore, in \cite{minavc2024isomorphic}, we extend this research by investigating whether the spectrum of $G_n(D)$ uniquely determines $D$,  answering an analogous conjecture of Sander-Sander for gcd-graphs over $\F_q[x]$.  Additionally, in \cite{nguyen2024integral}, we generalize So's theorem by giving the necessary and sufficient conditions for a Cayley graph over a finite symmetric algebra to be integral. As a by-product of this work, we construct examples of finite symmetric algebras with arithmetic origins.

The goal of this article is to define and study the concept of gcd-graphs over arbitrary finite (commutative, associative, unital) rings. Our hope is to unify various constructions in the literature and lay the groundwork for this area of research. In particular, we generalize some existing results in \cite{klotz2007some,  minavc2024gcd, nguyen2024integral, saxena2007parameters} to this general setting. 

\subsection{Outline}
In \cref{sec:gcd-graphs} we introduce the notion of gcd-graphs over a finite ring which naturally generalizes some previous work in \cite{klotz2007some, minavc2024gcd}. We also discuss various equivalence conditions for a graph to be a gcd-graph. As a by-product, we describe explicitly the structure of the generating set in a gcd-graph.  In \cref{sec:connectedness}, we investigate the connectivity of a gcd-graph. In particular, we provide a sharp upper bound on the diameter of a gcd-graph which generalizes a theorem of Saxena, Severini, and Shparlinski in \cite{saxena2007parameters}. Finally, in \cref{sec:spectrum}, we describe explicitly the spectrum of a gcd-graph over a finite symmetric algebra. The main result of this section gives a unified treatment for the spectra of various gcd-graphs previously studied in the literature. 

\section{Gcd-graphs over a finite ring} \label{sec:gcd-graphs}
In this section, we introduce the notion of a gcd-graph defined over a finite ring that unifies the definitions in \cite{klotz2007some, minavc2024gcd}.  Let us first recall the definition of a Cayley graph defined over a finite ring. 
\begin{definition}
Let $R$ be a finite ring and $S \subset R \setminus \{0 \}$ a symmetric subset. The Cayley graph $\Gamma(R,S)$ is defined as follow. 
\begin{enumerate}
    \item The vertex set of $\Gamma(R,S)$ is $R.$
    \item Two vertices $x, y \in R$ are adjacent if $x-y \in S.$
\end{enumerate}
In practice, $S$ is often called the generating set for $\Gamma(R,S).$ 
\end{definition}
As noted in \cite[Section 4]{chudnovsky2024prime} and supported further by \cite{minavc2024gcd, nguyen2024integral}, Cayley graphs defined over a ring exhibit richer structures compared to those defined over abstract abelian groups. This feature arises from the interaction between the additive and multiplicative structures of the ring $R$. In particular, ideals play a fundamental role in the study of these graphs.

For a ring that is not necessarily a quotient of a principal ideal domain, the notion of the greatest common divisor is not well defined. As a result, to define gcd-graphs over such a ring, we first need to revisit the definition of the greatest common divisor. We recall that a positive integer $n$ and two integers $a,b$, $\gcd(a,n) = \gcd(b,n)$ if and only if $a$ and $b$ generate the same ideal in the ring $\Z/n.$ This is also equivalent to the condition that $a \equiv ub \pmod{n}$ for some $u \in (\Z/n)^{\times}.$ This property generalizes well for Artinian rings, and in particular, to finite rings. More precisely 
\begin{lem} (See \cite[Lemma 2.1]{kaplansky1949elementary} \label{lem:kaplansky}
Let $R$ be an Artinian ring. If $Ra = Rb$, then there exists $u \in R^{\times}$ such that $b= ua.$ 
\end{lem}
From this perspective, a crucial observation here is that a generating set $S$ in a finite ring $R$ such as $\Z/n$ or $\F_q[x]/f$ gives rise to a gcd-graph if and only if $S$ is stable under the action of $R^{\times}$; that is, if $s \in S$, then $us \in S$ for all $u \in R^{\times}$. Motivated by this observation, we introduce the following definition.

\begin{definition} \label{def:gcd}
    We say that $\Gamma(R,S)$ is a gcd-graph if $S$ is stable under the action of $R^{\times}.$
\end{definition}
\begin{rem}

When $S=R^{\times}$, the graph $\Gamma(R,R^{\times})$ is known as a unitary Cayley graph (see \cite{unitary,klotz2007some}). This shows that the unitary Cayley graph over $R$ is a special case of gcd-graphs. Other examples, as explained previously, include the gcd-graphs defined in \cite{klotz2007some} for the ring $\Z/n$ and the gcd-graphs defined in \cite{minavc2024gcd} for the ring $\F_q[x]/f$ where $\F_q$ is a finite field with $q$ elements and $f$ is a non-zero polynomial in $\F_q[x].$
\end{rem}

We provide below the necessary and sufficient conditions for a Cayley graph over $R$ to be a gcd-graph. These conditions are somewhat more explicit than \cref{def:gcd}. Furthermore, together with \cref{cor:description_of_S}, this description will be important later on when we study the spectrum of these gcd-graphs. 

\begin{prop} \label{prop:gcd-graph}
Let $R$ be a finite commutative ring and $S$ a subset of $R$.
Then the following are equivalent. 
\begin{enumerate}
    \item $\Gamma(R,S)$ is a gcd-graph.
    \item There exist distinct nonzero principal ideals $\I_1, \I_2, \ldots, \I_k$ such that for each $r \in R$, $r \in S$ if and only if there exists $1 \leq i \leq k$ such that $\I_i = Rr.$
\end{enumerate}
\end{prop}

\begin{proof}
    First, we claim that $(2) \implies (1).$ Clearly, $S$ is symmetric and $0\not\in S$.
    By definition, we need to show that if $s \in S$ and $u \in R^{\times}$ then $us \in S$. Since $s \in S$, there exists an ideal $\I_i$ with $1 \leq i \leq k$ such that $Rs= \I_k.$ Since $u \in R^{\times}$, $\I_k = Rs = Rus.$ By $(2)$, this shows that $us \in S.$ 
    
    Let us show that $(1)$ implies $(2)$ as well. For each $s \in S$, Let $\I_s = Rs$ be the ideal generated by $s$. Let $\{\I_1, \I_2, \ldots, \I_k\}$ be the set of all $\I_s.$ We conclude that for each $s \in S$, $\I_s = \I_i$ for some $1 \leq i \leq k.$ Conversely, if $r \in R$ such that $Rr \in \{\I_1, \I_2, \ldots, \I_k\}$, then $Rr=Rs$ for some $s \in S.$ By \cref{lem:kaplansky}, we know that $r=us$ for some $u \in R^{\times}$. Since $\Gamma(R,S)$ is a gcd-graph, $S$ is stable under the action of $R^{\times}.$ This shows that $r \in S$ as well. 
\end{proof}
By \cref{prop:gcd-graph} and to be consistent with the literature, we will denote the gcd-graph $\Gamma(R,S)$ by $G_{R}(D)$ where $D = \{\I_1, \I_2, \ldots, \I_k \}.$

\begin{figure}[H]
\centering
\includegraphics[width=0.7 \linewidth]{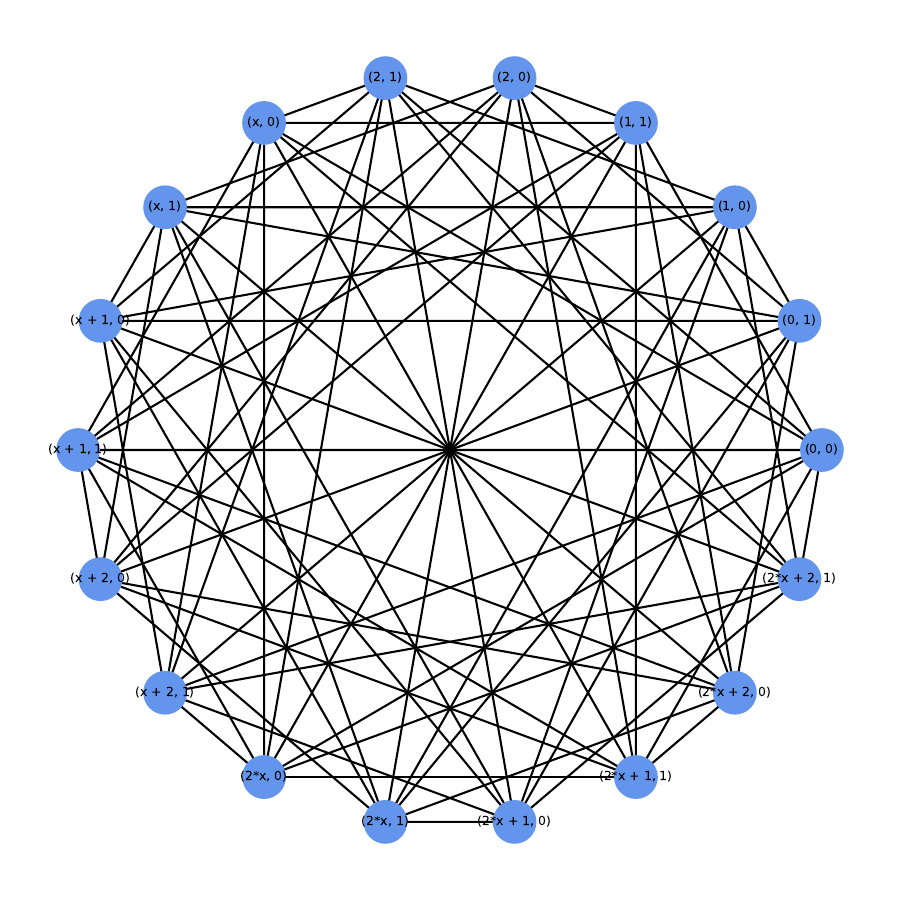}
\caption{The gcd-graph $G_{R}(D)$ where $R=\F_3[x]/x^2 \times \F_2$, $D = \{R(1,1), R(x,0)\}$}
\label{fig:opitmal}
\end{figure}

Our next goal is to provide an explicit description of the generating set $S$ for a gcd-graph $\Gamma(R,S).$ To do so, we first recall the following definition from ring theory.

\begin{definition}
    Let $T$ be a subset of $R.$ The annihilator ideal $\text{Ann}_{R}(T)$ is defined as 
    \[ \text{Ann}_R(T) = \{a \in R \mid at=0 \text{ for all } t \in T\}. \]
\end{definition}

With this definition, we can now explicitly describe the generating set $S$. 
\begin{cor} \label{cor:description_of_S}
Let $\Gamma(R,S) = G_{R}(D)$ be a gcd-graph where $D = \{\I_1, \I_2, \ldots, \I_k \}$. Suppose that $\I_i = Rx_i$.  Then every element $s \in S$ can be written  in the form $s = \hat{u} x_i$ for some $1 \leq i \leq k$, $u \in (R/{\rm Ann}_{R}(\I_i))^{\times}$, and $\hat{u}$ is a lift of $u$ in $R^{\times}$, and $u$ is uniqutely determined by $s$.
\end{cor}
\begin{proof}
    By \cref{prop:gcd-graph}, every $s \in S$ can be written in the form $s=\hat{u} x_i$ for some $1 \leq i \leq k$ and $\hat{u} \in R^{\times}.$ By definition, if $u_1 \equiv u_2 \pmod{{\rm Ann}_{R}(\I_i)}$ then $u_1x_i = u_2x_i.$ Therefore, the expression $s=\hat{u}x_i$ only depends on the class of $\hat{u} \in (R/{\rm Ann}_{R}(\I_i))^{\times}.$
\end{proof}

While the requirement that each $\I_i$ is principal seems quite strict, we will show below that finite rings which are quotients of the ring of integers in a global field have this property. To state and prove this statement, we first fix some notations and set up the background. Let $K$ be a global field. Let $\OO_K$ be the integral closure in $K$ of $\Z$ if $\text{char}(K)=0$ or of $\F_q[t]$ if $\text{char}(K)=p>0$. Let $\mathfrak{a}$ be a nonzero ideal in $\OO_K$. 
\begin{lem} \label{lem:principal_ring}
Let $R = \OO_K/\mathfrak{a}$. Then $R$ is a principal ring; i.e., all of its ideals are principal.  
\end{lem}
\begin{proof}
Let $\mathfrak{a} = \prod_{i=1}^d \mathfrak{p}_i^{e_i}$ be the factorization of $\mathfrak{a}$ into a product of distinct prime ideals. Then 
\[ R = \OO_K/\mathfrak{a} \cong \prod_{i=1}^d \OO_K/\mathfrak{p_i}^{e_i} \cong \prod_{i=1}^d (\OO_{K})_{\mathfrak{p_i}}/\mathfrak{p_i}^{e_i} .\] 
Here $(\OO_K)_{\mathfrak{p_i}}$ is the completion of $\OO_K$ at $\mathfrak{p_i}.$ It is known that $(\OO_K)_{\mathfrak{p}_i}$ is a discrete valuation ring, and in particular, a principal ideal domain. This implies that each factor $(\OO_{K})_{\mathfrak{p_i}}/\mathfrak{p_i}^{e_i}$ is a principal ideal ring. Consequently, $\OO_K/\mathfrak{a}$ is a principal ideal ring as well. 
\end{proof}

\begin{rem}
By \cite[Proposition 2.1]{nguyen2024integral}, if $R$ is a finite ring and $\Gamma(R,S)$ is a gcd-graph, then $\Gamma(R,S)$ is integral. In \cref{sec:spectrum}, we will provide an explicit description of the spectrum of a gcd-graph when $R$ is a finite symmetric algebra using the theory of generalized Ramanujan sums. 
\end{rem}

\section{Connectivity of gcd-graphs} \label{sec:connectedness}
In this section, we study the connectivity of a gcd-graph $\Gamma(R,S).$ By \cref{prop:gcd-graph}, we can assume that $\Gamma(R,S)=G_{R}(D)$ where 
$D = \{\I_1, \I_2, \ldots, \I_k \}$ is a set of principal ideals in $R$; namely $\I_i = Rx_i$ for some $x_i \in R.$   In this setting, the generating set $S$ is precisely the set of $s$ such that $Rs = \I_i$ for some $1 \leq i \leq k.$ Equivalently, there exists $u \in R^{\times}$ such that $s = ux_i.$

We start with the simplest case which is a direct generalization of \cite[Theorem 3.2]{minavc2024gcd}. We remark that the proof that we will give for this statement is not optimal in the sense that it does not provide a sharp upper bound for the diameter of $G_{R}(D).$ However, we include it here because it serves as a prototype for our argument in the general case. 
\begin{prop} \label{prop:case_G_R_connected}
    Assume that the unitary Cayley graph $G_{R}$ is connected. Then $G_{R}(D)$ is connected if and only if 
    \[ \mathcal{I}_1+ \cdots+ \mathcal{I}_k = R.\]
    Here $\mathcal{I}_1+ \cdots +\mathcal{I}_k$ is the sum of these ideals. 
\end{prop}

\begin{proof}
    Suppose that $G_{R}(D)$ is connected. Then $S$ generates $R$ as an abelian group. In particular, we can find $n_1, n_2, \ldots, n_h \in \Z$ and $s_i \in S$ such that
    \[ 1 = n_1 s_1 + n_2 s_2 +\cdots+ n_h s_h.\]
Because $Rs_i \in \{\I_1, \I_2, \ldots, \I_k\}$, the above equation shows that $1 \in \I_1+\I_2+\cdots +\I_k.$ Since this sum is an ideal, we conclude that $\I_1+\I_2+\cdots+\I_k=R.$ 

Conversely, suppose that $\I_1+\I_2+\cdots + \I_k =R$. We claim that $G_{R}(D)$ is connected. Let $ a \in R.$ By our assumption, we can find $a_1, a_2, \ldots, a_k \in R$ such that $a = \sum_{i=1}^k a_i x_i.$  Since $G_{R}$ is connected, for each $1 \leq i \leq k$, we can write $a_i = \sum_{j=1}^{n_i} m_{ij} s_{ij}$, where $m_{ij} \in \Z$ and $s_{ij} \in R^{\times}.$ Consequently, we can write 
    \[ a = \sum_{i=1}^k \sum_{j=1}^{n_i} m_{ij} s_{ij} x_i. \]
    Because $s_{ij} \in R^{\times}$, $s_{ij}x_i \in S.$ This shows that $a$ belongs to the abelian group generated by $S$. Since this is true for all $a \in R$, we conclude that $G_{R}(D)$ is connected. 
\end{proof}

\begin{rem}
    We can optimize the proof given in \cref{prop:case_G_R_connected} by carefully controlling the number of units $s_{ij}$ that sum up to $a_i.$  By definition, this number is bounded above by the diameter of $G_{R}.$ The above argument shows that when $G_{R}(D)$ is connected, its diameter is bounded above by $|D| \text{diam}(G_{R}).$ We remark that by \cite[Theorem 3.1]{unitary}, when $G_R$ is connected, its diameter is at most $3$; namely $\text{diam}(G_R) \leq 3.$ Consequently, we have a simple estimate ${\rm diag}(G_R(D)) \leq 3 |D|.$ We also note while this upper bound is explicit, it is not optimal in general. We refer to \cref{thm:connectitivy_main} for a better upper bound. 
\end{rem}

We now modify the above proof to the general case. We start with the following lemma. 

\begin{lem} \label{lem:quotient_property}
    Let $R$ be a finite ring and $R'$  a quotient of $R$; namely there exists a surjective ring homomorphism $\Phi\colon R \to R'$. Let $D = \{\I_1, \I_2, \ldots, \I_k\}$ be a set of principal ideals in $R$ and 
    \[ D' = \{\Phi(\I_1),\Phi(\I_2),  \ldots, \Phi(\I_k) \},  \]
    be the image of $D$ in $R'$ (to avoid the tautological case, we adopt the convention of removing $\I_i$ from $D'$ whenever $\Phi(\I_i) = 0$). Then the following statements hold.
    \begin{enumerate}
    \item $\Phi(\I_i)$ is a principal ideal for each $1 \leq i \leq k.$
    \item Let $S'$ be the generating set in $R'$ associated with $D'$ as described in \cref{prop:gcd-graph}. Then  $\Phi(S) \subseteq S'$. Consequently, $\Phi\colon  G_{R}(D) \to G_{R'}(D')$ is a graph morphism. 
    \item Suppose that $G_{R}(D)$ is connected. Then $G_{R'}(D')$ is also connected. 
    \end{enumerate}
\end{lem}

\begin{proof}
Let us first prove part $(1)$. Suppose that $\I_i =Rx_i$ for each $1 \leq i \leq k.$ Since $\Phi$ is surjective, we conclude that $\Phi(\I_i)= R' \Phi(x_i).$ This shows that $\Phi(\I_i)$ is a principal ideal generated by $\Phi(x_i).$

For the second part, we know from \cref{cor:description_of_S} that each $s \in S$ can be written in the form $s=ux_i$ for some $u \in R^{\times}.$ We then see that $\Phi(s)=\Phi(u) \Phi(x_i).$ Since $\Phi(u) \in (R')^{\times}$, we conclude that $R'\Phi(s) =R'\Phi(x_i)$ and hence  $\Phi(s) \in S'.$ This shows that $\Phi\colon G_{R}(D) \to G_{R'}(D')$ is a graph homomorphism. 

For the last part, since $\Phi\colon G_{R}(D) \to G_{R'}(D')$ is a graph homomorphism, $\Phi$ maps a walk in $G_{R}(D)$ to a walk in $G_{R'}(D').$ As a result, if $G_{R}(D)$ is connected, then $G_{R'}(D')$ is also connected. 
\end{proof}
We now deal with the general case. As observed in \cite[Theorem 3.1]{unitary} and \cite[Lemma 4.33]{chudnovsky2024prime}, the obstruction for $G_R$ to be connected is the existence of multiple local factors of $R$ whose residues are $\F_2.$ Let us explain in more detail. By the structure theorem, $R \cong \prod_{i=1}^d R_i$ where each $R_i$ is a local ring. For simplicity, let us write $R=R_1 \times R_2$ where $R_1$ (respectively $R_2$) consists of local factors whose residue fields are $\F_2$ (respectively $\neq \F_2$). Let $J(R_1)$ be the Jacobson radical of $R_1$. Because $R_1$ is a finite product of local rings, $J(R_1)$ is the kernel of the map $R_1 \to R_1/J(R_1) \cong \F_2^r$ where $r$ is the number of local factors whose residue fields are $\F_2.$ Keeping the same notation, we have the following lemma. 

\begin{lem} (See also \cite[Theorem 3.1]{unitary}) \label{lem:sum_two_units}
    Let $(T, \mathfrak{m})$ be a local ring whose residue field $T/\mathfrak{m}$ is not $\F_2.$ Then every element in $T$ can be written as the sum of two units. 
\end{lem}
\begin{proof}
    Let $a \in T$ and $\bar{a}$ the image of $a$ in $T/\mathfrak{m}.$ Since $T/\mathfrak{m}$ has more than $2$ elements, we can write $\bar{a}=\bar{u}_1+\bar{u}_2$ where $\bar{u}_1, \bar{u}_2 \neq 0.$ Consequently, there exist $u_1, u_2 \in T$ and $m \in \mathfrak{m}$ such that $a=u_1 + u_2+m = u_1 +(u_2+m).$ By their definition, $u_1, u_2+m \in T^{\times}.$ We conclude that every element in $T$ can be written as a sum of two units. 
\end{proof}

\begin{cor} \label{cor:sum_two_units}
    Every element in $J(R_1) \times R_2$ can be written as the sum of two units in $R.$
\end{cor}

\begin{proof}
    Let $(m,a) \in J(R_1) \times R_2$. By \cref{lem:sum_two_units}, we can write $a=u_1+u_2$ where $u_1, u_2 \in R_2^{\times}.$ Let $2^w$ be the characteristic of $R_1$. We then have 
    \[ (m,a)=(1,u_1)+(2^w-1+m, u_2).\]
    We remark that $2^w-1+m \in R_1^{\times}$ since its image in $R_1/J(R_1)  = \F_2^r$ is  $1$ which is a unit. Consequently, $(2^w-1+m, u_2) \in R^{\times}.$ We conclude that $(m,a)$ is the sum of two units in $R.$
\end{proof}

Keeping the same notation as above, we are now ready to state and prove the following theorem which is a direct generalization of \cite[Theorem 4]{saxena2007parameters}.

\begin{thm} \label{thm:connectitivy_main}
Let $\Phi: R \to R_1/J(R_1)=\F_2^r$ be the quotient ring homomorphism described above. Then $G_{R}(D)$ is connected if and only if the following two conditions hold 
\begin{enumerate}
    \item $\I_1 + \I_2 +\cdots +\I_k = R$;
    \item The cubelike graph $G_{\F_2^r}(D')$ is connected where 
    \[ D' = \{\Phi(\I_1), \Phi(\I_2), \ldots, \Phi(\I_k) \}. \] 
\end{enumerate}
Furthermore, suppose that the above conditions hold. Let $t$ be the smallest value of $t$ such that there exists $1 \leq i_1 < i_2 <\cdots < i_t \leq k$ such that 
\[ \I_{i_1} + \I_{i_2} + \cdots + \I_{i_t} = R.\]
Then 
\[ t \leq {\rm diam}(G_{R}(D)) \leq 2t + {\rm diam} (G_{\F_2^r}(D')).\]
\end{thm}

\begin{proof}
By \cref{lem:quotient_property}, we know that $(1)$ and $(2)$ are necessary conditions. We will show that they are sufficient as well. In fact, we will simultaneously show that $G_{R}(D)$ is connected and ${\rm diag}(G_{R}(D)) \leq 2t + {\rm diam} (G_{\F_2^r}(D'))$. We remark that since the only unit in $\F_2^{r}$ is $1$, the generating set for $G_{\F_2}(D')$ is precisely the set  $S'=\{\Phi(x_1), \Phi(x_2), \ldots, \Phi(x_k)\}$ where $\I_k=Rx_k.$ Let $a \in R$ be an arbitrary element in $R.$ We claim that the distance from $a$ to $0$ is at most $2t + {\rm diag} (G_{\F_2^r}(D'))$. In other words, we need to show that $a$ can be written as the sum of at most $2t + {\rm diam} (G_{\F_2^r}(D'))$ elements in $S.$

Since $G_{\F_2}(D')$ is connected, we can write $\Phi(a)= \sum_{i=1}^k n_i \Phi(x_i)$ 
where $n_i \in \{0,1\}$ and $\sum_{i=1}^k n_i \leq {\rm diam}(G_{\F_2^r})(D').$ Let $b:=a-\sum_{i=1}^k n_i x_i \in J(R_1) \times R_2.$ By our assumption $\I_{i_1} + \I_{i_2} + \cdots + \I_{i_t} = R$, we can write $1 = \sum_{m=1}^t a_{i_m} x_{i_m}.$ As a result, we can write $b = \sum_{m=1}^t (ba_{i_m}) x_{i_m}.$ By \cref{cor:sum_two_units}, for each $1 \leq m \leq t$, we can write $(ba_{i_m})$ as a the sum of two units in $R$. This shows that $b$ can be written as a sum of at most $2t$ elements in $S.$ Consequently, $a$ can be written as a sum of at most $2t + {\rm diam} (G_{\F_2^r}(D'))$ elements in $S.$ Since this is true for all $a \in R$, we conclude that 
\[ {\rm diam}(G_{R}(D)) \leq 2t + {\rm diam} (G_{\F_2^r}(D')). \]
The lower bound $t \leq {\rm diam}(G_{R}(D))$ follows from a similar argument. 
\end{proof}

\begin{rem}
    The proof for \cref{thm:connectitivy_main} shows $G_{R}(D)$ and $G_{\F_2^r}(D')$ have the same number of connected components. 
\end{rem}

\begin{rem}
    The estimate in \cref{thm:connectitivy_main} is sharp. For example, the upper bound holds for the graph $G_R(D)$ described in \cref{fig:opitmal} (namely for $R=(\F_3[x]/x^2) \times \F_2$ and $D =\{Rx_1, Rx_2\}$ where $x_1= (1,1)$ and $x_2= (x,0).$) In this case $t=1$ and the diameter of $G_{R}(D)$ is $3 = 2t+1.$
\end{rem}
In the special case where $R=\Z/n$, our theorem recovers the following estimate in \cite[Theorem 4]{saxena2007parameters}. 
\begin{cor}
    Suppose that $r \leq 1.$ Then $G_{R}(D)$ is connected if and only if $\I_1+\I_2 +\cdots+\I_k=R$. Furthermore, let $t$ be the smallest value of $t$ such that there exists $1 \leq i_1 < i_2 <\cdots < i_t \leq k$ such that 
\[ \I_{i_1} + \I_{i_2} + \cdots + \I_{i_t} = R.\]
We then have ${\rm diam}(G_{R}(D)) \leq 2t + 1.$ 
\end{cor}

\begin{proof}
    If $\I_1+\I_2+\cdots+\I_k=R$ then $D' \neq \emptyset.$ Since $r \leq 1$, this shows that $G_{\F_2^r}(D')$ is connected and its diameter is at most $1.$ \cref{thm:connectitivy_main} then shows that ${\rm diam}(G_{R}(D)) \leq 2t + 1.$
\end{proof}

\begin{rem}
    In \cite{bavsic2024maximal}, the authors determine the maximum diameter in the family of gcd-graphs over $\Z/n$ for fixed $n.$ It would be interesting to study the same problem for gcd-graphs over an arbitrary ring. 
\end{rem}

\section{Spectrum of $G_{R}(D).$} \label{sec:spectrum} 

In this section, we describe the spectrum of a gcd-graph $G_{R}(D)$ over a finite symmetric $\Z/n$-algebra. We remark that some authors refer to this as a  Frobenius ring (see \cite{honold2001characterization} and the references therein for some further studies on this class of rings). As noted in \cite{honold2001characterization}, this concept was independently studied in \cite{lamprecht1953allgemeine} by Lamprecht in his investigations of generalized Gauss sums. As we will see later, it will be evident to the readers that our work on the spectra of gcd-graphs owes much to the pioneering work of Lamprecht. We first recall this definition. 

\begin{definition}
Let $R$ be a finite $\Z/n$-algebra. We say that $R$ is symmetric if there exists an $\Z/n$-linear functional $\psi\colon R \to \Z/n$ such that the kernel of $\psi$ does not contain any non-zero ideal in $R.$ We call such $\psi$ a non-degenerate linear functional on $R$.
\end{definition}
For the rest of this article, we will assume that $R$ is a finite symmetric $\Z/n$-algebra equipped with a fixed linear functional $\psi\colon  R \to \Z/n$.  
Character theory for the additive group structure of $R$ is quite simple. More specifically, by \cite[Propsition 2.3]{nguyen2024integral}, for each character $\widehat{\psi} \in \Hom(R, \C^{\times})$ of $R$, there exists a unique element $r\in R$ such that for all $t \in R$
\[ \widehat{\psi}(t) = \zeta_n^{\psi_r(t)} = \zeta_n^{\psi(rt)}.\]
Here $\zeta_n \in \C$ is a fixed primitive $n$th root of unity. Let us recall the following standard lemma which we will need later on. 
\begin{lem} \label{lem:char_sum_standard}
Let $G$ be a finite abelian group and $\widehat{\psi}\colon G \to \C^{\times}$ be a nontrivial character of $G.$ Then $\sum_{g \in G} \widehat{\psi}(g)=0.$ 
\end{lem}
Let $x \in R$ and $\psi_x\colon  R/\text{Ann}_{R}(x) \to \Z/n$ be the linear functional defined by 
\[ \psi_x(\bar{a}) = \psi(ax), \]
here $\bar{a} \in R/\text{Ann}_{R}(x)$ and $a$ is any lift of $\bar{a}$ in $R.$ We can see that this map is well-defined. Furthermore, we have the following. 

\begin{lem} \label{lem:symmetric-quotient}
$\psi_x$ is a non-degenerate linear functional on $R/{\rm Ann}_{R}(x)$. Consequently, $R/\Ann_{R}(x)$ is a symmetric $\Z/n$-algebra. 
\end{lem}

\begin{proof}
Let $R'= R/\Ann_{R}(x)$. Suppose that $\psi_x$ is degenerate. Then, there exists $b' \in R'$ such that $b' \neq 0$ and $R'b' \subset \ker(\psi_x).$ Let $b$ be a lift of $b'$ to $R$. Then, by the definition of $\psi_x$ we have $\psi(Rbx)=0.$ This implies that the $\ker(\psi)$ contains the ideal $Rbx.$ Since $\psi$ is non-degenerate, $bx=0$ and hence $b \in \Ann_{R}(x).$ This would imply that $b' = 0$, which is a contradiction. 
\end{proof}

\begin{rem} \label{rem:construction}
It is not true that if $R$ is a finite symmetric algebra, then $R/\I$ is a symmetric algebra for every ideal $\I$ of $R.$ Let us provide a concrete example (see \cref{cor:quotient_symmetric} for a more general statement). Let $p$ be a prime number. We claim that $R = \F_p[x,y]/(x^2,y^2)$ is a symmetric $\F_p$-algebra. In fact, every element in $r \in R$ can be written in the form 
\[ r = a_0 +a_1x+a_2 y + a_3 xy.\]
We define a linear functional $\psi\colon R \to \F_p$ by $\psi(r)=a_3.$ We can check that this is a non-degenerate $\F_p$-linear functional on $R.$ On the other hand, the quotient $\F_p[x,y]/(x,y)^2$ of $R$ is not a symmetric $\F_p$-algebra. In fact, suppose that $\sigma\colon \F_p[x,y]/(x,y)^2 \to \F_p$ is an $\F_p$-linear functional on $\F_p[x,y]/(x,y)^2.$ If $\sigma(x)=\sigma(y)=0$ then $\ker(\sigma)$ contains the ideal $(x,y).$ Otherwise, $\ker(\sigma)$ contains the non-zero ideal generated by $\sigma(x)y-\sigma(y)x.$ This shows that, in all cases, $\sigma$ is degenerate. 
\end{rem}
We also remark that the construction of the symmetric algebra mentioned in \cref{rem:construction} could be generalized. We have the following observation. 

\begin{prop} \label{prop:construction_symmetric}
    Let $R$ be a finite symmetric $\Z/n$-algebra. Let $f \in R[x]$ be a monic polynomial of degree $n$. Then $R[x]/f$ is also a finite symmetric $\Z/n$-algebra. 
 \end{prop}
\begin{proof}
    Let $\psi\colon R \to \Z/n$ be a non-degenerate $\Z/n$-linear functional on $R.$ We will now define a non-degenerate $\Z/n$-linear functional on $R[x]/f.$ Each element of $R[x]/f$ can be written uniquely as 
    \[ g = a_{n-1}x^{n-1}+\cdots+a_1x+a_0.\]
This shows that $R[x]/f$ is a finite ring of order $|R|^n.$ We define $\widehat{\psi}\colon R[x]/f \to \Z/n$ by the rule $\widehat{\psi}(g) = \psi(a_{n-1}).$ By an identical argument as the proof of \cite[Proposition 6.7]{minavc2024gcd}, we can see that $\widehat{\psi}$ is non-degenerate. We conclude that $R[x]/f$ is a finite symmetric $\Z/n$-algebra. 
\end{proof}

\begin{rem}
    We recall that a Galois ring is a ring of the form $R=\Z[x]/(p^n, f(x)) = (\Z/p^n)[x]/f(x)$ where $f$ is a monic polynomial. \cref{prop:construction_symmetric} shows that Galois rings $  (\Z/p^n)[x]/f(x)$  are finite symmetric $\Z/p^n$-algebras. 
\end{rem}
Another corollary of \cref{prop:construction_symmetric} is the following.
\begin{cor} \label{cor:quotient_symmetric}
    Every finite commutative ring is a quotient of a finite symmetric algebra. 
\end{cor}
\begin{proof}
Let $T$ be a finite commutative ring and $n$ the characteristic of $T.$ Since $T$ is finite, there exists $\alpha_1, \alpha_2, \ldots, \alpha_k$ such that $T$ is generated by $\alpha_1, \alpha_2, \ldots, \alpha_k$; namely $T=(\Z/n)[\alpha_1, \alpha_2, \ldots, \alpha_k].$ Let us prove by induction that $(\Z/n)[\alpha_1, \alpha_2, \ldots, \alpha_i]$ is a quotient of a finite symmetric $\Z/n$-algebra for each $0 \leq i \leq k$. If $i=0$ then $T=\Z/n$ which is a symmetric $\Z/n$-algebra. Suppose that $(\Z/n)[\alpha_1, \alpha_2, \ldots, \alpha_i]$ is a quotient of a symmetric $\Z/n$-algebra, say $R.$ We claim that $(\Z/n)[\alpha_1, \alpha_2, \ldots, \alpha_{i+1}]= (\Z/n)[\alpha_1, \alpha_2, \ldots, \alpha_{i}][\alpha_{i+1}]$ is a quotient of a symmetric $\Z/n$-algebra as well. Since $T$ is finite, there exists a monic polynomial $f \in (\Z/n)[\alpha_1, \alpha_2, \ldots, \alpha_{i}][x]$ such that $f(\alpha_{i+1})=0.$ Let $\widehat{f}$ be a lift of $f$ to $R.$ Then we have the following quotient maps 
\[ R[x]/\widehat{f} \to (\Z/n)[\alpha_1, \alpha_2,\ldots, \alpha_i][x]/f \to (\Z/n)[\alpha_1, \alpha_2, \ldots, \alpha_{i}][\alpha_{i+1}].\]
By \cref{prop:construction_symmetric}, $R[x]/\widehat{f}$ is a finite symmetric $\Z/n$- algebra. This shows that the inductive statement holds for $i+1.$ By the principle of mathematical induction, $T$ is a quotient of a finite symmetric $\Z/n$-algebra. 
\end{proof}

We now define generalized Ramanujan sum. 
\begin{definition}
Let $g \in R$.  The generalized Ramanujan sum $c(g, R)$ is defined as follows
\[ c(g,R) = c_{\psi}(g, R) = \sum_{a \in R^{\times}} \zeta_n^{\psi(ga)}. \]
\end{definition}

Our goal is to give an explicit description for $c(g,R)$. In particular, we will show that $c(g,R)$ does not depend on the choice of $\psi$ as long as $\psi$ is non-degenerate.  Similar to the case with Ramanujan sums over $\Z$ as described in \cref{eq:ramanujan_2}, doing so would require some generalization of the M\"obius and Euler totient functions. The definition for the Euler function $\varphi(\I)$ is quite straightforward.

\begin{definition}
Let $T$ be a finite ring. The Euler number of $T$ is defined as 
\[ \varphi(T) = |T^{\times}|, \]
where $T^{\times}$ is the set of invertible elements in $T.$
\end{definition}

The definition of the M\"obius function $\mu(T)$ is a little more complicated. First, we recall that by the structure theorem for Artinian rings, the finite ring $T$ is isomorphic to a finite product of local rings $T \cong \prod_{i=1}^d R_i$.  The following definition is inspired by the classical M\"obius
 function.

\begin{definition}
\[
\mu(T) =
\begin{cases} 
    1, & \text{if } |T|=1, \\ 
    0, & \text{if there exists } 1 \leq i \leq d \text{ such that } R_i \text{ is not a field,} \\
    (-1)^d, & \text{otherwise.}
\end{cases}
\]
\end{definition}

\begin{expl}
Let us consider the case $R$ is a finite quotient of the ring of integers in a global field $K$; i.e, $R= \OO_K/\mathfrak{a}$ where $\mathfrak{a}$ is a non-zero ideal in $\OO_K$.  Suppose that $\mathfrak{a} = \prod_{i=1}^d \mathfrak{p}_i^{e_i}$ is the factorization of $\mathfrak{a}$ into a product of prime ideals, then 
\[ \OO_K/\mathfrak{a}  \cong \prod_{i=1}^d \OO_K/\mathfrak{p}_i^{e_i}.\]
By definition, $\mu(\OO_K/\mathfrak{a})=0$ if there exists $1 \leq i \leq d$ such that $e_i>1.$ Otherwise, $\mu(\OO_K/\mathfrak{a})=(-1)^d.$ With this interpretation, we can see the $\mu(\OO_K/\mathfrak{a})$ is a direct generalization of the classical M\"obius function. 
\end{expl}

We discuss a simple property for the behavior of the Euler and M\"obius functions with respect to direct products. 
\begin{lem} \label{lem:property_mobius}
    Let $R$ be a finite ring. Suppose that $R=R_1 \times R_2$. Then $\mu(R) = \mu(R_1) \mu(R_2)$ and $\varphi(R) = \varphi(R_1) \varphi(R_2).$
\end{lem}

\begin{proof}
Both statements follow directly from the definition of $\mu$ and $\varphi.$
\end{proof}

With these preparations, we can now calculate the generalized Ramanujan sum $c(g, R).$ We first calculate the Ramanujan sum $c(g,R)$ when $g=1$. 
\begin{prop} \label{prop:value_at_1} $c(1,R) = \mu(R).$
\end{prop}

\begin{proof}
By the structure theorem, $R$ is isomorphic to a product of local rings; i.e., $R \cong \prod_{i=1}^d R_i$, where $R_i$ is a local ring. Let $\psi_i$ be the linear functional on $R_i$ induced by $\psi.$ Since $\psi$ is non-degenrate, $\psi_i$ is non-degenerate as well. Furthermore, by \cite[Satz1]{lamprecht1953allgemeine} we have $c_{\psi}(1, R) = \prod_{i=1}^d c_{\psi_i}(1, R_i).$ Together with \cref{lem:property_mobius} about the multiplicative property of the M\"obius function under direct products, it is sufficient to prove the statement when $R$ is a local ring. Namely, we need to show that if $R$ is a local ring then 
\[ c_{\psi}(1,R) = \begin{cases} 
    0, & \text{if } R \text{ is not a field,} \\
    -1 & \text{otherwise.}
\end{cases}\]
Let $\mathfrak{m}$ be the maximal ideal of $R$ and $\widehat{\psi} = \zeta_n^{\psi}$ be the chacteracter of $R$ associated with $\psi.$ Because $R^{\times} = R \setminus \mathfrak{m}$, we have 
\begin{equation} \label{eq:local_calculation}
c_{\psi}(1,R) = \sum_{a \in R^{\times}} \widehat{\psi}(a) = \sum_{a \in R} \widehat{\psi}(a) - \sum_{a \in \mathfrak{m}} \widehat{\psi}(a). 
\end{equation}
Since $\widehat{\psi}$ is a nontrivial character of $R$, we know that $\sum_{a \in R} \widehat{\psi}(a) =0$. Additionally, because $\mathfrak{m}$ is an additive subgroup of $R$, the restriction of $\widehat{\psi}$ to $\mathfrak{m}$ is a character of $\mathfrak{m}$ (considered as an abstract abelian group). Since $\psi$ is non-degenerate, the restriction of $\widehat{\psi}$ to $\mathfrak{m}$ is a non-trivial character unless $\mathfrak{m}=0.$ By \cref{lem:char_sum_standard}, we conclude that 
\[ \sum_{a \in \mathfrak{m}} \widehat{\psi}(a) = \begin{cases} 
    0, & \text{if } \mathfrak{m} \neq 0 \\
    1 & \text{otherwise.}
\end{cases}\]
By \cref{eq:local_calculation}, we conclude that if $R$ is a local ring then 
\[ c_{\psi}(1,R) = \begin{cases} 
    0, & \text{if } R \text{ is not a field,} \\
    -1 & \text{otherwise.}
\end{cases}\]
\end{proof}

We now consider the general case. 

\begin{thm} \label{prop:evaluate_at_g}
    \[ c(g, R) = \frac{\varphi(R)}{\varphi(R/\Ann_{R}(g))} c(1, R/\Ann_R(g)) = \frac{\varphi(R)}{\varphi(R/\Ann_{R}(g))} \mu(R/\Ann_R(g)).  \]
\end{thm}
\begin{proof}
By definition  $c(g, R) = \sum_{a \in R^{\times}} \zeta_n^{\psi(ga)}.$ 
We remark that if $a-b \in \Ann_{R}(g)$ then $ga=gb$ and hence $\psi(ga)=\psi(gb).$ We also note that since the reduction map $\Phi\colon R^{\times} \to (R/\Ann_R(g))^{\times}$ is a surjective group homomorphism, each $u \in (R/\Ann_R(g))^{\times}$ has exactly $\frac{\varphi(R)}{\varphi(R/\Ann_R(g))}$ lifts to $R^{\times}.$ Therefore 

\[ c(g, R) = \frac{\varphi(R)}{\varphi(R/\Ann_R(g))} \sum_{a \in (R/\Ann_{R}(g))^{\times}} \zeta_n^{\psi(ga)} = \frac{\varphi(R)}{\varphi(R/\Ann_R(g))} \sum_{a \in (R/\Ann_{R}(g))^{\times}} \zeta_n^{\psi_g(a)} . \] 

By \cref{lem:symmetric-quotient}, $\psi_g$ is a non-degenerate linear functional on $R/\Ann_{R}(g).$ By \cref{prop:value_at_1}, we know that 
\[ \sum_{a \in (R/Ann_{R}(g))^{\times}} \zeta_n^{\psi_g(a)}  = \mu(R/\Ann_R(g)).\] 
We conclude that 
    \[ c(g, R) = \frac{\varphi(R)}{\varphi(R/\Ann_{R}(g))} \mu(R/\Ann_R(g)).  
    \qedhere\]
\end{proof}
\begin{lem}
    Let $g,x\in R$ and $R'=R/\Ann_R(x)$. Then $R/\Ann_R(gx)\simeq R'/\Ann_{R'}(g')$, here $g'$ is the image  of $g$ in $R'$. In particular, $c(g',R')=\frac{\varphi(R/\Ann_R(x))}{\varphi(R/\Ann_{R}(gx))} \mu(R/\Ann_R(gx))$.
\end{lem}
\begin{proof} Consider the quotient map $\varphi\colon R\to R'/\Ann_{R'}(g')$, which is the composiiton of two quotient maps $R\to R'$ and $R'\to R'/\Ann_{R'}(g')$. For any $r\in R$, $r\in \ker\varphi$ if and only if $g'r'=0\in R'$ (here $r'$ is the image of $r$ in $R'$) if and only if $gr\in \Ann_R(x)$ if and only if $grx=0$ if and only if $r\in \Ann_R(gx)$. Hence $\ker \varphi= \Ann_R(gx)$ and $\varphi$ induces an isomomorphism $R/\Ann_R(gx)\simeq R'/\Ann_{R'}(g')$.
\end{proof}
We can now describe explicitly the spectrum of a gcd-graph. 
\begin{thm} \label{thm:main}
Let $G_{R}(D)$ be a gcd-graph with $D=\{\I_1, \I_2, \ldots, \I_k \}$ and $\I_i = Rx_i$ is a principal ideal. Then, the spectrum of $G_{D}(D)$ is the multiset $\{\lambda_g \}_{g \in R}$  where 
\[ \lambda_g = \sum_{i=1}^k c(g, R/\Ann_{R}(x_i)). \]
Here $c(g, R/\Ann_{R}(x_i))$ is the Ramanujan sum with an explicit formula given in \cref{prop:evaluate_at_g}. More precisely 
\[ c(g, R/\Ann_{R}(x_i)) = \frac{\varphi(R/\Ann_{R}(x_i))}{\varphi(R/\Ann_{R}(gx_i))} \mu(R/\Ann_R(gx_i))\]

\end{thm}

\begin{proof}
Let $S$ be the generating set associated with $D$ as described in \cref{prop:gcd-graph}. By the circulant diagonalization theorem, the spectrum of $G_{R}(D)= \Gamma(R,S)$ is the multiset $\{\lambda_g \}_{g \in R}$ where 
\[ \lambda_g = \sum_{s \in S} \zeta_n^{\psi(gs)} = \sum_{i=1}^k \left[\sum_{s, Rs = \I_i} \zeta_n^{\psi(gs)} \right]. \]

We remark that by \cref{cor:description_of_S}, if $s \in R$ such that $Rs = \I_i =Rx_i$ then $s$ has a unique representation of the form $s = \hat{u}x_i$ where $u \in (R/\Ann_{R}(x_i))^{\times}$ and $\hat{u}$ is a fixed lift of $u$ to $R^{\times}.$ With this presentation, we can write 
\[ \sum_{s, Rs = \I_i} \zeta_n^{\psi(gs)} = \sum_{u \in (R/\Ann_{R}(x_i))^{\times}} \zeta_n^{\psi(gu x_i)} = \sum_{u \in (R/\Ann_{R}(x_i))^{\times}} \zeta_n^{\psi_{x_i}(gu)} = c(g, R/\Ann_{R}(x_i)). \]
Here we recall that $\psi_{x_i}$ is the induced linear functional on $R/\Ann_{R}(x_i).$ We conclude that $\lambda_g = \sum_{i=1}^k c(g, R/\Ann_{R}(x_i)).$
\end{proof}

The following corollary is simple yet important for our future work on perfect state transfers on gcd-graphs.
\begin{cor}
    Suppose that $g'=ug$ for some $u \in R^{\times}.$ Then $\lambda_{g}=\lambda_{g'}.$
\end{cor}

\section*{Acknowledgements}
We thank the Department of Mathematics and Computer Science at Lake Forest College for their generous financial support through an Overleaf subscription. We also thank  J\'an Min\'a\v{c} for his constant encouragement and support.   

\bibliographystyle{amsplain}
\bibliography{references.bib}

\providecommand{\bysame}{\leavevmode\hbox to3em{\hrulefill}\thinspace}
\providecommand{\MR}{\relax\ifhmode\unskip\space\fi MR }
\providecommand{\MRhref}[2]{%
  \href{http://www.ams.org/mathscinet-getitem?mr=#1}{#2}
}
\providecommand{\href}[2]{#2}
\begin{thebibliography}{10}

\bibitem{unitary}
Reza Akhtar, Megan Boggess, Tiffany Jackson-Henderson, Isidora Jim{\'e}nez, Rachel Karpman, Amanda Kinzel, and Dan Pritikin, \emph{On the unitary {Cayley} graph of a finite ring}, Electron. J. Combin. \textbf{16} (2009), no.~1, Research Paper 117, 13 pages.

\bibitem{bavsic2024maximal}
Milan Ba{\v{s}}i{\'c}, Aleksandar Ili{\'c}, and Aleksandar Stamenkovi{\'c}, \emph{Maximal diameter of integral circulant graphs}, Information and Computation \textbf{301} (2024), 105208.

\bibitem{chudnovsky2024prime}
Maria Chudnovsky, Michal Cizek, Logan Crew, J{\'a}n Min{\'a}{\v{c}}, Tung~T. Nguyen, Sophie Spirkl, and Nguy{\^e}n~Duy T{\^a}n, \emph{On prime {Cayley} graphs}, arXiv:2401.06062, to appear in Journal of Combinatorics (2024).

\bibitem{honold2001characterization}
Thomas Honold, \emph{Characterization of finite frobenius rings}, Archiv der Mathematik \textbf{76} (2001), no.~6, 406--415.

\bibitem{kaplansky1949elementary}
Irving Kaplansky, \emph{Elementary divisors and modules}, Transactions of the American Mathematical Society \textbf{66} (1949), no.~2, 464--491.

\bibitem{klotz2007some}
Walter Klotz and Torsten Sander, \emph{Some properties of unitary {Cayley} graphs}, The Electronic Journal of Combinatorics \textbf{14} (2007), no.~1, R45, 12 pages.

\bibitem{lamprecht1953allgemeine}
Erich Lamprecht, \emph{Allgemeine theorie der {Gau{\ss}schen} {Summen} in endlichen kommutativen {Ringen}}, Mathematische Nachrichten \textbf{9} (1953), no.~3, 149--196.

\bibitem{minavc2024isomorphic}
J{\'a}n Min{\'a}{\v{c}}, Tung~T Nguyen, and Nguyen~Duy T{\^a}n, \emph{Isomorphic gcd-graphs over polynomial rings}, arXiv preprint arXiv:2411.01768 (2024).

\bibitem{minavc2024gcd}
\bysame, \emph{On the gcd graphs over polynomial rings}, arXiv preprint arXiv:2409.01929 (2024).

\bibitem{nguyen2024integral}
Tung~T Nguyen and Nguyen~Duy T{\^a}n, \emph{Integral cayley graphs over a finite symmetric algebra}, Archiv der Mathematik, to appear.

\bibitem{saxena2007parameters}
Nitin Saxena, Simone Severini, and Igor~E Shparlinski, \emph{Parameters of integral circulant graphs and periodic quantum dynamics}, International Journal of Quantum Information \textbf{5} (2007), no.~03, 417--430.

\bibitem{so2006integral}
Wasin So, \emph{Integral circulant graphs}, Discrete Mathematics \textbf{306} (2006), no.~1, 153--158.

\end{thebibliography}
\end{document}